\documentclass[a4paper,12pt]{article}
\usepackage{amsmath,amssymb,amsthm,graphics,latexsym,amsfonts}
\usepackage{fancyhdr}
\usepackage{color}
\usepackage{graphics}
\title{\Large 
On the maximum value of conflict-free verex-connection number of graphs\thanks{Research supported by NSFC (No. 11571294)}}
\author{ {Zhenzhen Li, Baoyindureng Wu \footnote{Corresponding author.
Email: baoywu@163.com (B. Wu) }}\\
\small  College of Mathematics and System Sciences, Xinjiang
University \\ \small  Urumqi, Xinjiang 830046, P.R.China \\
}
\date{}

\newtheorem{theorem}{Theorem}[section]
\newtheorem{lemma}[theorem]{Lemma}

\newtheorem{conjecture}[theorem]{Conjecture}

\usepackage{indentfirst}

\begin{document}
\maketitle {\small \noindent{\bfseries Abstract} A path in a vertex-colored graph is called {\it conflict-free} if there is a color used on exactly one of its vertices. A vertex-colored graph is
said to be {\it conflict-free vertex-connected} if any two vertices of the graph
are connected by a conflict-free path. The {\it conflict-free vertex-connection number},
denoted by $vcfc(G)$, is defined as the smallest number of colors required to make $G$
conflict-free vertex-connected. Li et al. \cite{Li} conjectured that for a connected graph $G$ of order $n$, $vcfc(G)\leq vcfc(P_n)$. We confirm that the conjecture is true and pose a a relevant conjecture concerning the conflict-free connection number introduced by Czap et al. in \cite{Czap}.
\\
{\bfseries Keywords}: Conflict-free connection; Conflict-free vertex-connection; Tree

\section {\large Introduction}
We consider simple, finite and undirected graphs only, and refer to the book \cite{Bondy} for undefined notation and terminology. Let $G=(V,E)$ be a finite graph with vertex set $V$ and edge set
$E$. The {\it size} of $G$, denoted by $e(G)$, is $|E|$. The {\it degree}
of a vertex $v$, denoted by $d_G(v)$, is the number of edges which
are incident with $v$ in $G$. As usual, $\delta(G)$ and $\Delta(G)$
denote the minimum degree and the maximum degree of $G$,
respectively. A subgraph $H$ of $G$ is a {\it spanning subgraph} of $G$ if $V(H)=V(G)$.
We use $K_n$, $P_n$, and $K_{1, n-1}$ to denote the complete graph, the path, and the star of order $n$, respectively.

Very recently, Czap et al. introduced the concept of conflict-free connection in \cite{Czap}. A path in an edge-colored graph is called {\it conflict-free} if there is a color used on exactly one of its edges. An edge-colored graph is said to be {\it conflict-free connected} if any two vertices of the graph are connected by a conflict-free path. The {\it conflict-free connection number} of a connected graph, denoted by $cfc(G)$, is defined a the smallest number of colors required to make $G$ conflict-free connected.

Motivated by the above mentioned concepts, as a natural counterpart of a conflict-free connection number, Li et al. \cite{Li} introduced the concept of conflict-free vertex-connection number. A path in a vertex-colored graph is called {\it conflict-free} if there is a color used on exactly one of its vertices. A vertex-colored graph is
said to be {\it conflict-free vertex-connected} if any two vertices of the graph
are connected by a conflict-free path. The {\it conflict-free vertex connection number},
denoted by $vcfc(G)$, is defined as the smallest number of colors required to make $G$
conflict-free vertex-connected. Note that for a nontrivial connected graph $G$ of order $n$, \begin{equation}
2\leq vcfc(G)\leq n
\end{equation}

Li et al. determined the conflict-free vertex connection number of almost all graphs by showing the following result.

\begin{theorem}( Li et al. \cite{Li})
Let $G$ be a connected graph $G$ of order at least three. Then $vcfc(G)=2$ if and only if $G$ is 2-connected or it has only one cut vertex.
\end{theorem}

So, a basic question arise: {\it what is the maximum value of the conflict-free vertex connection numbers of all graphs of order $n$ ?}

\vspace{2mm} It can be observed in \cite{Li} that for a nontrivial connected graph $G$, if $H$ is a spanning subgraph of $G$,
then $vcfc(H)\geq vcfc(G)$. In particular, for any spanning tree $T$ of $G$, $vcfc(T)\geq vcfc(G)$. Thus the maximum value of the conflict-free vertex connection numbers must be achieved by some tree of order $n$. It can be checked that $vcfc(K_{1, n-1})=2$ for any $n\geq 2$. In particular, Li et al. showed that

\begin{theorem}(Li et al. \cite{Li} )
For an integer $n\geq 2$, $vcfc(P_n)=\lceil log_2(n+1)\rceil$.
\end{theorem}

A {\it $k$-ranking} of a connected graph $G$ is a labeling of its vertices with labels $1, \ldots, k$ such that every path between any two vertices with the same label $i$ in $G$ contains at least one vertex with label $j>i$. A graph $G$ is said to be $k$-rankable if it has a $k$-ranking. The minimum $k$ for which $G$ is $k$-rankable is denoted by $r(G)$.
Iyer et al. \cite{Iyer} showed that
for a tree of order $n\geq 3$, $r(T)\leq log_{\frac 3 2} \ n$.
Li et al. \cite{Li} showed that

\begin{theorem} (Li et al. \cite{Li} )
For a tree of order $n\geq 3$, $vcfc(T)\leq log_{\frac 3 2} \ n$.
\end{theorem}
Further, Li et al. \cite{Li} conjectured that
\begin{conjecture}
For a connected graph $G$ of order $n$, $vcfc(G)\leq vcfc(P_n)$.
\end{conjecture}

The aim of this note is to prove the conjecture. We refer to \cite{Chang1, Chang2, Che, Chei, Even, Pach} for some relevant works on conflict-free coloring of graphs.

\section{\large The proof}

We begin with the following key lemma. For convenience, we denote by $moc(T-v)$ the maximum value of the orders of all components of $T-v$.
\begin{lemma} Let $T$ be a tree of order $n\geq 3$. If $n$ is odd, then there exists
a vertex $v$ with $moc(T-v)\leq \frac {n-1} 2$.
\end{lemma}

\begin{proof}

\vspace{2mm}

Choose a vertex $v_0\in V(T)$ such that $moc(T-v_0)=\min\{moc(T-v)|\ v$ run over all non-leaf vertices of $T\}$. We claim that $v_0$ is the vertex $v$, as we required. If it is not, then there exists a component of $T-v_0$, say $T_1$, has order $n_1>\frac {n-1} 2$. Note that $n_1=moc(T-v_0)$. Let $v_1$ be the neighbor of $v_0$ in $T_1$.
Let us consider the orders of the components of $T-v_1$. The component of $T-v_0v_1$ containing $v_0$ is a component of $T-v_1$ having order with $n-n_1<n-\frac {n-1} 2=\frac {n+1} 2$ (implying that $n-n_1\leq \frac {n-1} 2$).
Moreover, since all other components of $T-v_1$ is a proper subgraph of $T_1$, their orders are less than the order of $T_1$. It follows that $moc(T-v_1)<n_1=moc(T-v_0)$, contradicting the choice of $v_0$. This shows that the claim is true, and thus the result follows.

\end{proof}

Now we are ready to prove Conjecture 1.4.
\begin{theorem} For a tree $T$ of order $n$, $vcfc(T)\leq vcfc(P_n)$.
\end{theorem}
\begin{proof}
We show it by induction on $n$.
The result is trivially true when $n=2$, and now assume that $n\geq 3$. Then there
exists an integer $k\geq 2$ such that $2^{k-1}\leq n\leq 2^k-1$. Let $T$ be a spanning tree of $G$. As we have seen before, $vcfc(G)\leq vcfc(T)$. By Theorem 1.2, $vcfc(P_n)=k$. So it suffices to show that $vcfc(T)\leq k$.

By Lemma 2.1, there exists a vertex $v\in V(G)$ with $moc(T-v)\leq 2^{k-1}-1$ (If necessary, by adding some pendent vertices to $T$, one can make the resulting tree $T'$ have $2^{k-1}-1$ vertices, and apply Lemma 2.1 to $T'$). Let $T_1, \ldots, T_l$ be all components of $T-v$, and $n_i=|V(T_i)|$ for each $i$. By the induction hypothesis, $vcfc(T_i)\leq k-1$ for each $i$.
Taking a conflict-free coloring of $T_i$ using colors in $\{1, \ldots, k-1\}$ for each $i$, and color the vertex $v$ by $k$, we obtain a conflict-free coloring of $T$ using colors in $\{1, \ldots, k\}$. This proves $vcfc(T)\leq k$, and thus $vcfc(G)\leq vcfc(P_n)$.
\end{proof}

\section{\large Further research}

In this note, we focuss on the conflict-free vertex-connection number, and combining some known results, we have shown that for a connected graph of order $n$, $$2\leq vcfc(G)\leq \lceil log_2(n+1)\rceil,$$ where the lower bound can be achieved by 2-connected graphs of order $n$ and the upper bound can be achieved by $P_n$.

In \cite{Czap}, Czap et al. determined the conflict-free connection number of all graphs by showing that for a noncomplete 2-connected graph $G$, $cfc(G)=2$.
It was further extended in \cite{Chang1} by Chang et al.  showing that
for a noncomplete 2-edge-connected graph $G$, $cfc(G)=2$.

Clearly, for an integer $n\geq 2$, $K_n$ is the unique connected graph $G$ of order $n$
with $cfc(G)=1$. On the other hand, $cfc(K_{1,n-1})=n-1$. Observe that for a nontrivial connected graph $G$, if $H$ is a spanning subgraph of $G$,
then $cfc(H)\geq cfc(G)$. In particular, for any spanning tree $T$ of $G$, $cfc(T)\geq cfc(G)$. Thus the maximum value of the conflict-free connection numbers must be achieved by some tree of order $n$. Actually, (see \cite{Li}) for a nontrivial connected graph $G$ of order $n$, $$1\leq cfc(G)\leq n-1,$$ with the left hand side of equality if and only if $G\cong K_n$, and with the right hand side of equality if and only if $G\cong K_{1, n-1}$.

It is an interesting problem to decide that among all trees of order $n$, which one has the least conflict-free connection number ? Czap et al. \cite{Czap} showed that
for an integer $n\geq 2$, $cfc(P_n)=\lceil log_2\ n\rceil$.
We pose the following conjecture.

\begin{conjecture}
For a tree $T$ of order $n$, $cfc(T)\geq \lceil log_2\ n\rceil$.
\end{conjecture}

Another interesting problem, posed by Li \cite{Li2}, is the complexity for determining the conflict-free connection number or the conflict-free vertex-connection number of a graph.
Yang \cite{Yang} designed a polynomial-time algorithm to determine the conflict-free connection number of a tree.




\begin{thebibliography}{111}
\bibitem{Bondy} J.A. Bondy, U.S.R. Murty, Graph Theory, GTM 244, Springer, 2008.

\bibitem{Chang1} H. Chang, T.D. Doan, Z. Huang, X. Li, I. Schiermeyer,
Graphs wiht conflict-free connection number two, Xiv:1707.01634v1 [math.CO].

\bibitem{Chang2} H. Chang, Z. Huang, X. Li, Y. Mao, H, Zhao, Nordhaus-Gaddum-type theorem
for conflict-free connection number of graphs, arXiv:1705.08316 [math.CO].




\bibitem{Che} P. Cheilaris, B. Keszegh. D. P\'{a}lv\"{o}igyi,
Unique-maximum and conflict-free coloring for hypergraphs and tree graphs, SIAM J. Discrete Math. 27 (2013) 1775-1787.


\bibitem{Chei} P. Cheilaris, G. T\'{o}th,
Graph unique-maximum and conflict-free colorings, J. Discrete Algorithms 9 (2011) 241-251.

\bibitem{Czap} J. Czap, S. Jendrol¡¯,
J. Valiska, Conflict-free connection of graphs, Accepted by Discuss. Math. Graph Theory.


\bibitem{Deng} B. Deng, W. Li, X, Li, Y. Mao, H. Zhao, Conflict-free connection numbers of
line graphs, arXiv:1705.05317 [math.CO].

\bibitem{Even} G. Even, Z. Lotker, D. Ron, S. Smorodinsky, Conflict-free coloring of simple
geometic regions with applications to frequency assignment in cellular networks,
SIAM J. Comput. 33 (2003) 94-136.

\bibitem{Iyer} A.V. Iyer, H.D. Ratliff, G. Vijayan, Optimal node ranking of trees, Inform. Process.
Lett. 28 (1988) 225-229.

\bibitem{Li} X. Li, Y. Zhang, X. Zhu, Y. Mao, H. Zhao,
Conflict-free vertex-connections of graphs, arXiv:1705.07270v1[math.CO].

\bibitem{Li2} X. Li, private communication.

\bibitem{Pach} J. Pach, G. Tardos, Conflict-free colourings of graphs and hypergraphs, Comb.
Probab. Comput. 18 (2009) 819-834.

\bibitem{Yang} W. Yang, private communication.












\end{thebibliography}
\end{document}